\documentclass[12pt,draft]{article}
\usepackage{amsmath, amsthm, amssymb}
\usepackage[margin=1in]{geometry}

\newtheorem*{thm*}{Theorem}
\newtheorem*{cor*}{Corollary}
\newtheorem{lem}{Lemma}

\begin{document}

\nocite{*}
\title{Turbulent transport and coherent structures in 3D plasmas}

\author{Keith Leitmeyer}

\maketitle

\begin{abstract}
Kinetic and magnetic enstrophy are shown to concentrate to smaller scales from the integral scale to a Kraichnan-type scale for the 3D magnetohydrodynamic equations.  This is an improvement of the result from Bradshaw and Gruji\'c (2013), using redesigned ensemble averages.
\end{abstract}
\section{Introduction}
Observational and numerical evidence suggests that in magnetohydrodynamic (MHD)
turbulence the vorticity $\omega = \nabla \times u$ and current $j = \nabla \times b$
(here, $u$ and $b$ are the velocity and the magnetic field, respectively)
concentrate on coherent quasi low-dimensional structures--predominantly quasi 
two-dimensional sheets--which become 
increasingly thin exhibiting
morphological dynamics consistent with a process known as \emph{turbulent cascade}.
Understanding this process is important due to its intimate relationship  with violent reconnection
events in the realm of \emph{solar wind turbulence},
and is highly relevant to the goals outlined
in the Space Studies Board of the National Research Council (NRC) survey,
22\emph{A Decadal Strategy for Solar and Space Physics (Heliophysics)}, completed in 2012, and in 
particular, to the strategic goal to ``discover and characterize fundamental processes that occur both 
within the heliosphere and throughout the universe''\cite{NRC}.

\medskip

As a matter of fact--in the last five years--there has been a flurry of activity in the heliophysics
community directed at understanding the phenomenon of \emph{turbulent dissipation}, especially
within the range of kinetic scales. One of the most promising theories, supported by
extensive computational simulations, is that the coherent structures (most notably \emph{current
sheets}) exhibit a process of turbulent cascade down to the kinetic scales, essentially, all the
way to electron scales, where they trigger extremely strong and \emph{localized}
heating of the plasma, dissipating the energy (the so-called \emph{heating via current sheets})\cite{Karimabadi,Farge,Zhdankin}. 
There is also a sense of optimism in part of the community that the relevance of this theory 
to the solar wind could be confirmed via the data to be collected by the upcoming
NASA Magnetospheric Multiscale (MMS) mission.

\medskip

The zero-step in this theory is an assumption on the geometry of the turbulent plasma
at the \emph{interface} between the continuum (described by the MHD system) and the 
kinetic (described by the Vlasov-Maxwell-Poisson system)-scale dynamics; shortly, the
predominance of the current sheet geometry originating at continuum scales is \emph{assumed},
and then utilized as the input when descending into the kinetic-scale dynamics. This brings
out the question of \emph{existence} of \emph{turbulent cascades}--and in particular,
the cascades of kinetic and magnetic enstrophies--in the \emph{continuum}/MHD regime to the
forefront of scientific interest.

Kinetic and magnetic enstrophy was shown to concentrate towards smaller scales in \cite{BG} using ensemble averages in physical scales.  Here the ensemble averages used are redesigned to allow weaker assumptions.  The ensemble averages used to state the enstrophy concentration theorem are described first.  The following section goes over the conditions under which enstrophy concentration is shown.  Then we recall the bounds obtained in \cite{BG}, and finally the enstrophy concentration theorem is proven.

Related work on cascades and locality in hydrodynamic turbulence can be found in \mbox{[2-5,\,7-13,\,15,\,16,\,18-21]}.

\section{Fluxes and ensemble averages}
The MHD equations, which model evolution of the velocity and magnetic fields in
an electrically conducting incompressible fluid, read
\begin{equation}
\begin{aligned}\label{MHD}
&\partial_t u - \Delta u + (u\cdot\nabla)u - (b\cdot\nabla)b+\nabla P =0,
\\
&\partial_t b - \Delta b + (u\cdot\nabla)b - (b\cdot\nabla)u=0,
\\
&\nabla\cdot u = \nabla \cdot b =0.
\end{aligned}
\end{equation} 
(Here, $P$ is the total pressure, and the magnetic resistivity and the kinematic viscosity
are normalized to 1.)

\medskip

Taking the curl of equation \ref{MHD} gives equations for vorticity and current,
\begin{equation}
\begin{aligned}\label{vorticity}
&\partial_t \omega - \Delta \omega = -(u\cdot\nabla)\omega + (\omega\cdot\nabla)u + (b\cdot\nabla)j - (j\cdot\nabla)b,
\\
&\partial_t j - \Delta j = -(u\cdot\nabla)j + (j\cdot\nabla)u +(b\cdot\nabla)\omega - (\omega\cdot\nabla)b + 2\sum _{l=1}^3 \nabla b_l \times \nabla u_l.
\end{aligned}
\end{equation} 

These equations can be used to study \emph{inward} kinetic and magnetic enstrophy fluxes.  Rather 
than traditional fluxes through the boundary of a ball $B=B(x_0,2R)$,

\begin{equation}
\begin{aligned}
&-\int_{\partial B} \frac{1}{2}|\omega|^2(u\cdot n)\,d\sigma = -\int_B (u\cdot\nabla)\omega\cdot\omega\,dx,
\\
&-\int_{\partial B} \frac{1}{2}|j|^2(u\cdot n)\,d\sigma = -\int_B (u\cdot\nabla)j\cdot j\,dx,
\end{aligned}
\end{equation} 

we use a smooth cutoff function $\psi$ supported on $B(x_0,2R)$ and equal to 1
on $B(x_0,R)$--with inward pointing gradient--and
study fluxes through the spherical layer of thickness $R$,

\begin{equation}
\begin{aligned}\label{flux}
&\int \frac{1}{2}|\omega|^2(u\cdot \nabla\psi)\,dx = -\int (u\cdot\nabla)\omega\cdot\psi\omega\,dx,
\\
&\int \frac{1}{2}|j|^2(u\cdot \nabla\psi)\,dx = -\int (u\cdot\nabla)j\cdot \psi j\,dx.
\end{aligned}
\end{equation} 

The reason is that this form of the flux is more amenable to mathematical analysis, while
at the same time preserving the physics.  Time-averaged quantities are studied in turbulence, so we take time averages weighted according to a smooth function $\eta(t)$ which is 0 on $[0,T/3]$ and 1 on $[2T/3,T]$, denoting $\phi(x,t)=\psi(x)\eta(t)$:

\begin{equation}
\begin{aligned}\label{timeavgflux}
&\int_0^T\int \frac{1}{2}|\omega|^2(u\cdot \nabla\phi)\,dx\,dt = -\int_0^T\int (u\cdot\nabla)\omega\cdot\phi\omega\,dx\,dt,
\\
&\int_0^T\int \frac{1}{2}|j|^2(u\cdot \nabla\phi)\,dx\,dt = -\int_0^T\int (u\cdot\nabla)j\cdot \phi j\,dx\,dt.
\end{aligned}
\end{equation}

Multiplying equation \ref{vorticity} by $\phi\omega$ and $\phi j$ respectively, and integrating over space and time, will yield expressions which can be used to \emph{dynamically} estimate the 
fluxes:

\begin{equation}
\begin{aligned}\label{wflux}
&\int_0^T\int\frac{1}{2}|\omega|^2(u\cdot\nabla\phi)\,dx\,dt = \int\frac{1}{2}|\omega(x,T)|^2\psi(x)\,dx + \int_0^T\int|\nabla\omega|^2\phi\,dx\,dt \\
& -\int_0^T\int\frac{1}{2}|\omega|^2(\partial_s\phi+\Delta \phi)\,dx\,dt -\int_0^T\int (\omega\cdot\nabla)u\cdot(\phi\omega)\,dx\,dt \\
&- \int_0^T\int (b\cdot\nabla)j\cdot(\phi\omega)\,dx\,dt + \int_0^T\int (j\cdot\nabla)b\cdot(\phi\omega)\,dx\,dt \\
& =\int\frac{1}{2}|\omega(x,T)|^2\psi(x)\,dx + \int_0^T\int|\nabla\omega|^2\phi\,dx\,dt + H^\omega+N^\omega_1+L^\omega+N^\omega_2
\end{aligned}
\end{equation}

\begin{equation}
\begin{aligned}\label{jflux}
&\int_0^T\int\frac{1}{2}|j|^2(u\cdot\nabla\phi)\,dx\,dt = \int\frac{1}{2}|j(x,T)|^2\psi(x)\,dx + \int_0^T\int|\nabla j|^2\phi\,dx\,dt \\
& -\int_0^T\int\frac{1}{2}|j|^2(\partial_s\phi+\Delta \phi)\,dx\,dt +\int_0^T\int (\omega\cdot\nabla)b\cdot(\phi j)\,dx\,dt \\
&- \int_0^T\int (b\cdot\nabla)\omega\cdot(\phi j)\,dx\,dt - \int_0^T\int (j\cdot\nabla)u\cdot(\phi j)\,dx\,dt \\
& -\int_0^T\int \Big( 2\sum_{l=1}^3 \nabla u_l \times \nabla b_l \Big) \cdot(\phi j) \,dx\,dt \\
& =\int\frac{1}{2}|j(x,T)|^2\psi(x)\,dx + \int_0^T\int|\nabla j|^2\phi\,dx\,dt + H^j+N^j_1+L^j+N^j_2+X
\end{aligned}
\end{equation}

Enstrophy concentration will be demonstrated locally, over a ball $B(0,2R_0)$.  Ultimately we want to show that the average of these fluxes over suitable collections of functions $\psi$ of a particular scale is positive, for a range of scales.

Fix $C_0>1$ and $3/4<\rho<1$.  A \textbf{refined test function} at scale $R$ is any smooth function $\psi$ supported in a ball of radius $2R$ satisfying $0\leq\psi\leq 1$, $|\nabla \psi|<\frac{C_0}{R} \psi^\rho$, and $|\Delta\psi|<\frac{C_0}{R^2} \psi^{2\rho-1}$.

Now fix a scale $R_0$ refined test function $\psi_0$ centered at 0. An \textbf{ensemble} at scale $R$ with global multiplicity $K_1$ and local multiplicity $K_2$ is a collection of scale $R$ test functions $\{\psi_i\}_{i=1}^n$ satisfying the following properties:
\begin{enumerate}
\item $\psi_i \leq \psi_0 \leq \sum \psi_i$
\item $(R_0/R)^3\leq n\leq K_1 (R_0/R)^3$
\item No point of $B(2R_0,0)$ is contained in more than $K_2$ of the supports of $\psi_i$.
\end{enumerate}

For a function $f$, denote by $\langle F \rangle _R$ the \textbf{ensemble average} $\frac{1}{n}\sum_{i=1}^n\frac{1}{R^3} \int f\psi_{i,R}\,dx$, and $F_0=\frac{1}{R_0^3}\int f \psi_0 \,dx$. 

  Property 1 above is needed to compare $\langle F\rangle_R$ to $F_0$. Due to Property 1, test functions near the boundary of the support of $\psi_0$ will have small integrals, effectively skewing the ensemble average towards zero.  Larger $K_1$ and $K_2$ allow ensembles that have higher weight on functions away from the boundary, making the skewing insignificant.

These ensemble averages can be viewed as a way to detect whether a function is significantly negative above some spatial scale.  If every ensemble average $\langle F \rangle _R$ is positive, no matter how one arranges and stacks the test functions, then the function is not significantly negative at scales larger than $R$.  Increasing $K_1$ and $K_2$ lowers the threshold for a function to be considered significantly negative.

One way to explicitly construct ensembles is to apply Lemma 2 to $\psi_0$ and varying the multiplicity of the resulting functions (Assume $\psi_0$ satisfies the stronger $C_0'$-bounds to get an ensemble with $C_0$-bounds).

The following lemma states that ensemble averages (at any scale) of positive functions are comparable to the large scale mean.  The proof immediately follows from the definitions.

\begin{lem}\label{lem1} If $f\geq 0$ then $\frac{1}{K_1}F_0 \leq \langle F\rangle_R \leq K_2 F_0$.  For slightly modified ensemble averages, we have $\frac{1}{n}\sum_1^n \frac{1}{R^3} \int f\psi_{i,R}^\delta\, dx \leq K_2\frac{1}{R_0^3} \int f\psi_0^\delta\, dx \;(\delta>0)$.\qed 
\end{lem}

Using a refined partition of unity, one can turn larger scale ensembles into smaller scale ensembles.

\begin{lem}\label{lem2}
Any scale $R$ test function satisfying $C_0$ bounds is a sum of $8\lceil R/R'\rceil^3$ scale $R'$ test functions satisfying $C'_0$ bounds (where $R>R'$ and $C_0'$ depends only on $C_0$).

Therefore for all $(K_1,K_2,C_0)$-ensembles at scale $R$ and every $R'<R$, there exists a $(64K_1,8K_2,C_0')$-ensemble at scale $R'$ such that $\langle F\rangle_R=\langle F \rangle_{R'}$.
\end{lem}
\begin{proof}
Let $\psi$ be a scale $R$ test function satisfying $C_0$ bounds .  Now to construct the partition of unity, take a scale $R'$ test function $g_0$ (satisfying $C_0$ bounds), centered at zero and equal to 1 on $[-R',R']^3$.  Define $g_p=g_0(x-2R'p)$, where $p\in \mathbb{Z}^3$.  Then $1\leq\sum_p g_p\leq 2$ so we may define $h_p= g_p / \sum_q g_q$.

Some calculus shows that $|\nabla h_p | < \frac{6C_0}{R'} h_p^\rho$ and $|\Delta h_p|<\frac{3C_0+10C_0^2}{R'^2} h_p^{2\rho-1}$, so $|\nabla (\psi h_p)|< \frac {7C_0}{R'} (\psi h_p)^\rho$ and $|\Delta (\psi h_p)|<\frac{4C_0+22C_0^2}{R'^2}(\psi h_p)^{2\rho-1}$. Fewer than $8\lceil R/R'\rceil^3\leq 64(R/R')^3$ of the functions $\psi h_p$ are nonzero, and for any $x$, $\psi_p(x)\neq 0$ for at most 8 functions.

Since $\psi = \sum_p \psi h_p$, the first claim is proven.  For the second claim, given an ensemble $\{\psi_i\}_i$, the new ensemble will be $\{\psi_ih_p\}_{i,p}$.
\end{proof}
\section{Assumptions}

Now we come to the conditions on the current and vorticity over $B(0,2R_0)\times(0,T)$ under which we can show that there is enstrophy concentration.

\subsection{Geometry/smoothness}\label{geosmooth}

Denote by $\theta(u,v)$ the angle between two vectors $u, v$.  It is required that for some $M, C_1>0$, 

\begin{align}
| \sin \theta(\omega(x+y,t), \omega(x,t))|\leq C_1|y|^{1/2} \label{geo}
\end{align}

for every $t\in(0,T)$, every $x\in B(0,2R_0+R_0^{2/3})$ with $|\nabla u(x,t)|>M$, and every $y$ such that $|y|< 2 (\sigma_0/\beta) +(\sigma_0/\beta)^{2/3},$ and
\begin{align}
|j(x+y,t)-j(x,t)|\leq |j(x+y,t)||y|^{1/2} \label{smooth}
\end{align}
for every $t\in(0,T)$, every $x\in B(0,2R_0+R_0^{2/3})$ with $|\nabla b(x,t)|>M$, and every $y$ such that $|y|< 2 (\sigma_0/\beta) +(\sigma_0/\beta)^{2/3}.$  The quantity $\sigma_0/\beta$ is defined in section \ref{kraichnan}.

Condition \ref{geo} depends only on the angle of the vorticity vector $\omega$.  This $\frac{1}{2}$--Holder coherence will deplete the vortex stretching term.  This condition is needed because condition \ref{smooth} is insufficient to control the vortex stretching term $N^\omega_1$, which has no explicit dependence on the magnetic field.

Condition \ref{smooth} is needed for the nonlinear terms that do not have a geometric kernel available.  It requires $\frac{1}{2}$--Holder continuity, and more.  It would be too restrictive if $|j(x+y,t)|$ were ever small, but $x+y$ is always close to the region where $|\nabla b|$ is large, and roughly, $\nabla b$ and $j$ are large in the same regions.

In \cite{BG}, condition \ref{smooth} was used with $j$ replaced by $\omega$ with no assumption corresponding to \ref{geo}.  The present formulation is preferred because current appears to be more regular than vorticity in numerical simulations.

\subsection{Kraichnan-type scale}\label{kraichnan}
Let $e_0$, $E_0$, and $P_0$ denote the time-averaged total energy, total enstrophy, and total palinstrophy at the integral scale.  Precisely,

\begin{align*}
&e_0=\frac{1}{T}\int_0^T\frac{1}{R_0^3}\int \phi_0^{4\rho-3}\big(\frac{|u|^2}{2}+\frac{|b|^2}{2}\big)\,dx\,dt, \\
&E_0 = \frac{1}{T}\int_0^T\frac{1}{R_0^3}\int \phi_0^{2\rho-1}(|\omega|^2+|j|^2)\,dx\,dt, \\
&P_0 = \frac{1}{T}\int_0^T\frac{1}{R_0^3}\int \phi_0(|\nabla\omega|^2+|\nabla j|^2)\,dx\,dt + \frac{1}{TR_0^3}\int\frac{1}{2}(|\omega(x,T)|^2+|j(x,T)|^2)\psi_0\,dx.
\end{align*}
Define the modified Kraichnan-type scale $\sigma_0$ by 
\begin{align}\label{kscale}
\sigma_0= \text{max}\{\Big(\frac{E_0}{P_0}\Big)^{1/2} , \Big(\frac{e_0}{P_0}\Big)^{1/4} \}.
\end{align}

Assumption 2 is that $\sigma_0<\beta R_0$, where $\beta$ is a constant ($0<\beta<1$) identified in the proof.  If $\omega$ and $j$ have large gradients as expected in a turbulent flow, this assumption will be satisfied.
\subsection{Localization}\label{localization}
Any smooth solution to the equations will have $\omega \in L^2((0,T)\times B(0,2R_0+R_0^{2/3}))$.  It is required that the kinetic and magnetic enstrophy is not too highly concentrated around any point:
\begin{align} \label{eq:local}
\int_0^T\int_{B(y,R)}|\omega|^2+|j|^2\,dx\,dt < \frac{1}{C_2}
\end{align}
for any $y\in B(0,2R_0)$ where $C_1$ is a constant and $R=2\sigma_0/\beta+(\sigma_0/\beta)^{2/3}$.

In \cite{BG}, it was required that $$\int_0^T\int_{B(0,2R_0+R_0^{2/3})}|\omega|^2\,dx\,dt < \frac{1}{C_3}.$$  This assumption restricts the extent of the range of scales over which we can show enstrophy concentration.  The magnetic enstrophy is in the present assumption because of the geometric/smoothness assumption on the current that is not found in \cite{BG}.

\subsection{Modulation}
The assumption imposes a restriction on the time evolution of the integral-scale kinetic and magnetic enstrophies across $(0,T)$ consistent with our choice of the temporal cutoff. Precisely,
\begin{align*}
&\int|\omega(x,T)|^2\psi_0(x)\,dx \geq \frac{1}{2} \sup_t \int |\omega(x,t)|^2\psi_0(x)\,dx,\\
&\int|j(x,T)|^2\psi_0(x)\,dx \geq \frac{1}{2} \sup_t \int |j(x,t)|^2\psi_0(x)\,dx,
\end{align*}

\section{Bounds}
In \cite{BG}, the terms of equations \ref{wflux} and \ref{jflux} are bounded by quantities which can be related to $e_0, E_0,$ and $P_0$.

Obtaining the desired bounds using Assumption \ref{geosmooth} will use all of the same techniques as in \cite{BG} with the major difference being labelling, except for the vortex stretching term which uses geometric depletion as in \cite{DG2}.

Ultimately, we have
\begin{align*}
H^\omega+H^j&+N_1^\omega + N_2^\omega + N_1^j + N_2^j+ L^\omega + L^j+X \\
\leq &\,K_P\Big(\frac{1}{\alpha}+||\omega||_{L^2( (0,T)\times B(x_i,2R+R^{2/3}))} \Big)\Big(\frac{1}{2}\sup_{t\in(0,T)} \int \psi(x)(|\omega(x,t)|^2+|j(x,t)|^2)\,dx\, + \\&\int_0^T\int\phi(|\nabla\omega|^2+| \nabla j|^2)\,dx\,dt \Big) + \frac{K_E}{R^2}\int_0^T\int \phi^{2\rho-1}(|\omega|^2+|j|^2)\,dx\,dt\,+\\& \frac{\alpha^2K_e}{R^4}\int_0^T\int \phi^{4\rho-3} \frac{|u|^2+|b|^2}{2}\,dx\,dt,
\end{align*}

where $K_P, K_E, K_e$ are constants and $\alpha$ is an interpolation parameter that we may choose.  The constants $K_P, K_E, K_e$ are fixed upon a choice of the parameters $K_1, K_2, \rho, C_0, C_1,$ and $M$.  In the following section $K_P,K_E, K_e$ are instead the constants obtained corresponding to $64K_1, 8K_2,$ and $C_0'$.

\section{Main result}

Kinetic and magnetic enstrophy is on average being transported to smaller scales down to the Kraichnan-type scale $\sigma_0/\beta$.  The following notation will be used in the proof: For a density $f$ and ensemble $\{\psi_i\}_{i=1}^n$, the average over one element of the ensemble is $F_i=\int f \psi_i\,dx$ and the ensemble average is $\langle F \rangle_R= \frac{1}{n}\frac{1}{TR^3}\sum_{i=1}^n F_i$.  Let $\varphi=-\int_0^T (u\cdot\nabla)\omega\cdot\omega + (u\cdot\nabla) j\cdot j\,dt$, the enstrophy flux density, so that $\Phi_i=\int \varphi \psi_i\,dx$ is the kinetic and magnetic enstrophy flux centered at $x_i$ at scale $R$. Referring to equations \ref{wflux} and \ref{jflux}, suppressing the subscript $i$, denote $H=H^\omega+H^j$ and likewise for $N$ and $L$. Define similarly $$e=\int_0^T\int \phi_i^{4\rho-3}\frac{1}{2}(|u|^2+|b|^2)\,dt\,dt ,$$ $$E= \int_0^T\int \phi_i^{2\rho-1} (|\omega|^2 +|j|^2)\,dx\,dt,$$ $$P=\int \psi_i(x)(|\omega(x,T)|^2+|j(x,T)|^2)\,dx + \int_0^T\int\phi_i(|\nabla\omega|^2+| \nabla j|^2)\,dx\,dt,$$ and  $$\tilde{P}=\frac{1}{2}\sup_{t\in(0,T)} \int \psi_i(x)(|\omega(x,t)|^2+|j(x,t)|^2)\,dx + \int_0^T\int\phi_i(|\nabla\omega|^2+| \nabla j|^2)\,dx\,dt.$$

\begin{thm*}
Under Assumptions 1-4, for any $K_1$ and $K_2$ there exists $K_*$ such that for any $(K_1,K_2)$-ensemble at scale $R$ ranging from $\sigma_0/\beta$ to $R_0$, we have $\displaystyle \frac{1}{K_*}P_0 \leq\langle \Phi \rangle_R \leq K_* P_0.$ 
\end{thm*}

\begin{proof}

We start by showing that for $R=\sigma_0/\beta$, for $(64K_1,8K_2)$-ensembles satisfying $C_0'$ bounds, we have $\displaystyle \frac{1}{K_*}P_0 \leq\langle \Phi \rangle_R \leq K_* P_0.$

Referring to equations \ref{wflux} and \ref{jflux}, $\langle \Phi \rangle_R = \langle P \rangle_R + \langle H + N + L + X \rangle _R$, where $p$ is the positive density $\frac{1}{2}|\omega(T)|^2 + \int_0^T |\nabla\omega|^2\eta\,dt$ so that $\frac{1}{64K_1}P_0\leq \langle P\rangle_R \leq 8K_2 P_0$.  The rest of the terms are relatively small upon averaging: $\displaystyle |H+N+L+X|\leq K_P\Big(\frac{1}{\alpha}+||\omega||_{L^2(B(x_i,2R+R^{2/3})\times(0,T))} \Big)\tilde{P} + \frac{K_E}{R^2}E+ \frac{\alpha^2K_e}{R^4}e$, so that
\begin{align*}
|\langle H + N + L + X \rangle _R| &\leq \Big\langle K_P\big(\frac{1}{\alpha}+||\omega|| \big)\tilde{P} + \frac{K_E}{R^2}E+ \frac{\alpha^2K_e}{R^4}e \Big\rangle_R
\\& \leq 8K_2K_P\Big(\frac{1}{\alpha}+||\omega||\Big)\tilde{P_0}+\frac{8K_2K_E}{R^2}E_0+\frac{\alpha^2 8K_2K_e}{R^4}e_0
\\ &= \Big(8K_2K_P\big(\frac{1}{\alpha}+||\omega||\big) +8K_2K_E\beta^2 + \alpha^2 8K_2 K_e \beta^4 \Big)P_0
\\ & \leq \frac{3}{4\cdot 64K_1} P_0,
\end{align*}
by taking $\alpha= 4\cdot 64K_1\cdot 8K_2K_P$, using that $\beta$ is sufficiently small such that $8K_2K_E\beta^2+\alpha^2\cdot 8K_2K_e\beta^4 \leq \frac{1}{4\cdot 64K_1}$, and using the assumption that $||\omega||_{L^2(B(x_i,2R+R^{2/3})\times(0,T))}\leq\frac{1}{C_2}= \frac{1}{4\cdot64K_1\cdot 8K_2K_P}$.

Therefore $\displaystyle \frac{1}{4\cdot 64K_1}P_0 \leq \langle \Phi\rangle_R \leq (8K_2+ \frac{3}{4\cdot 64K_1} )P_0$ for any $(64K_1,8K_2, C_0')$-ensemble at scale $R=\sigma_0/\beta$.  For the range of scales from $\sigma_0/\beta$ up to $R_0$, by Lemma \ref{lem2},  $\displaystyle \frac{1}{4\cdot 64K_1}P_0 \leq\langle \Phi \rangle_R \leq (8K_2 +\frac{3}{4\cdot 64K_1}) P_0$ for all scale $R$ $(K_1,K_2, C_0)$-ensembles.

\end{proof}

The locality of magnetic and kinetic enstrophy flux stated in \cite{BG} also holds in our setting of more satisfying assumptions and modified ensemble averages.  The result is stated in terms of the time-averaged enstrophy flux, rather than the time-averaged enstrophy flux per unit mass used above.  That is, $$\langle \Psi \rangle_R = \frac{1}{n} \sum_{i=1}^n \frac{1}{T}\int_0^T \int \frac{1}{2}(|\omega|^2+|j|^2)(u\cdot\nabla\phi_i)\,dx\,dt$$ for an ensemble $\{\phi_i\}_{i=1}^n$ at scale $R$.
\begin{cor*}
Under Assumptions 1-4, for any $K_1, K_2$ there exists $K_*$ such that for any $r, R$ between $\sigma_0/\beta$ and $R_0$ and any $(K_1,K_2)$ ensembles, enstrophy flux is local:
\begin{align*}
\frac{1}{K_*^2} \Big(\frac{r}{R}\Big)^3 \leq \frac{\langle \Psi\rangle_r}{\langle \Psi\rangle_R} \leq K_*^2 \Big(\frac{r}{R}\Big)^3.
\end{align*}
\end{cor*}

\subsubsection*{Acknowledgments}
I would like to thank my advisor, Professor Gruji\'c, for his suggestions and guidance.

\bibliographystyle{plain}
\bibliography{bib}

\end{document}